\documentclass[11pt]{amsart}

\usepackage{geometry}                
\usepackage{graphicx}
\usepackage{amssymb}
\usepackage{rotating}

\usepackage{latexsym, amssymb, amsmath, amsthm, graphics}
\usepackage{url}
\usepackage{eucal} 
    
\begin{document}
\title{ \bf On the Image Conjecture}
\author{Arno van den Essen, David Wright, and Wenhua Zhao}



 \input amssym.def
    \input amssym
    \theoremstyle{definition} 
    \newtheorem{rema}{Remark}[section]
    \newtheorem{questions}[rema]{Questions}
    \newtheorem{assertion}[rema]{Assertion}
    \theoremstyle{plain} 
    \newtheorem{propo}[rema]{Proposition}
    \newtheorem{theo}[rema]{Theorem} 
        \newtheorem{claim}[rema]{Claim} 
    \newtheorem{conj}[rema]{Conjecture}
    \newtheorem{quest}[rema]{Question} 
    \theoremstyle{definition}
    \newtheorem{defi}[rema]{Definition}
    \theoremstyle{plain}
    \newtheorem{lemma}[rema]{Lemma} 
    \newtheorem{corol}[rema]{Corollary}
    \newtheorem{exam}[rema]{Example} 
    \newtheorem{rmk}[rema]{Remark}
    \newcommand{\del}{\triangledown}
    \newcommand{\nno}{\nonumber} 
    \newcommand{\lbar}{\big\vert}
    \newcommand{\mbar}{\mbox{\large $\vert$}} 
    \newcommand{\p}{\partial}
    \newcommand{\dps}{\displaystyle} 
    \newcommand{\bra}{\langle}
    \newcommand{\ket}{\rangle} 
    \newcommand{\kr}{\mbox{\rm Ker}\ }
    \newcommand{\res}{\mbox{\rm Res}} 
    \renewcommand{\hom}{\mbox{\rm Hom}}
    \newcommand{\pf}{{\it Proof:}\hspace{2ex}}
    \newcommand{\epf}{\hspace{2em}$\Box$}
    \newcommand{\epfv}{\hspace{1em}$\Box$\vspace{1em}}
    \newcommand{\nord}{\mbox{\scriptsize ${\circ\atop\circ}$}}
    \newcommand{\wt}{\mbox{\rm wt}\ } 
    \newcommand{\clr}{\mbox{\rm clr}\ }
    \newcommand{\ideg}{\mbox{\rm Ideg}\ } 
    \newcommand{\gC}{{\mathfrak g}_{\mathbb C}}
    \newcommand{\hatC}{\widehat {\mathbb C}} 
    \newcommand{\bZ}{{\mathbb Z}} 
    \newcommand{\bQ}{{\mathbb Q}}
    \newcommand{\bR}{{\mathbb R}} 
    \newcommand{\bN}{{\mathbb N}}
    \newcommand{\bT}{{\mathbb T}} 
    \newcommand{\fg}{{\mathfrak g}}
    \newcommand{\fgC}{{\mathfrak g}_{\bC}} 
    \newcommand{\cD}{\mathcal D} 
    \newcommand{\cP}{\mathcal P}
    \newcommand{\cC}{\mathcal C} 
    \newcommand{\cS}{\mathcal S}
     \newcommand{\cL}{\mathcal L}
    \newcommand{\EGC}{{\cal E}(\GC)}
    \newcommand{\cLGC}{\widetilde{L}_{an}\GC}
    \newcommand{\LGC}{{L}_{an}\GC} 
    \newcommand{\BQ}{\begin{eqnarray}}
    \newcommand{\EQ}{\end{eqnarray}} 
    \newcommand{\BQn}{\begin{eqnarray*}}
    \newcommand{\EQn}{\end{eqnarray*}} 
    \newcommand{\wtilde}{\widetilde}
    \newcommand{\Hol}{\mbox{Hol}} 
    \newcommand{\Hom}{\mbox{Hom}}
    \newcommand{\poly}{polynomial } 
    \newcommand{\polys}{polynomials }
    \newcommand{\pz}{\frac{\p}{\p z}} 
    \newcommand{\pzi}{\frac{\p}{\p z_i}}
    \newcommand{\edge}{\text{\raisebox{2.75pt}{\makebox[20pt][s]{\hrulefill}}}}
    \newcommand{\halfedge}{\text{\raisebox{2.75pt}{\makebox[10pt][s]{\hrulefill}}}}
    \newcommand{\n}{\notag}
    \newcommand{\C}{\mathbb C}
    \newcommand{\A}{\mathcal A}
    \newcommand{\Q}{\mathbb Q}
    \newcommand{\R}{\mathbb R}
    \newcommand{\X}{X_1,\ldots,X_n}
    \newcommand{\pa}{\partial}
    \newcommand{\D}{D_1,\ldots,D_n}
    \newcommand{\Del}{\text{\raisebox{2.1pt}{$\bigtriangledown$}}}
    \newcommand{\La}{\triangle}
    \newcommand{\Hess}{\text{\rm Hess}}
\newcommand{\F}{F_{1},\ldots,F_{n}}
\newcommand{\G}{G_{1},\ldots,G_{n}}
\newcommand{\GA}{\text{GA}}
\newcommand{\SA}{\text{SA}}
\newcommand{\MA}{\text{MA}}
\newcommand{\GL}{\text{GL}}
\newcommand{\SL}{\text{SL}}
\newcommand{\GE}{\text{GE}}
\newcommand{\Tr}{\text{Tr}}
\newcommand{\Af}{\text{Af}}
\newcommand{\Bf}{\text{Bf}}
\newcommand{\W}{\text{W}}
\newcommand{\EA}{\text{EA}}
\newcommand{\BA}{\text{BA}}
\newcommand{\E}{\text{E}}
\newcommand{\B}{\text{B}}
\newcommand{\T}{\text{T}}
\newcommand{\Di}{\text{D}}
\newcommand{\degr}{\text{deg}\,}
\newcommand{\supp}{\text{supp}\,}
\newcommand{\Z}{\mathbb Z}
\newcommand{\ideaala}{\mathfrak{a}}
\newcommand{\ideaalb}{\mathfrak{b}}
\newcommand{\ideaalc}{\mathfrak{c}}
\newcommand{\m}{\mathfrak{m}}
\newcommand{\id}{\text{id}}

\setcounter{equation}{0}
\setcounter{rema}{0}

\addtolength{\oddsidemargin}{.65cm}
\addtolength{\evensidemargin}{-.65cm}

\begin{abstract}
The Image Conjecture was formulated by the third author, who showed that it implied his Vanishing Conjecture, which is equivalent to the famous Jacobian Conjecture. We prove various cases of the Image Conjecture and show that how it leads to another fascinating and elusive assertion that we here dub the Factorial Conjecture. Various cases of the Factorial Conjecture are proved.
\end{abstract}

\thanks{The research of the third author was partially supported by NSA Grant H98230-10-1-0168}
    \maketitle

\section{Introduction} 

The notion of a Mathieu subspace was introduced by coauthor Wenhua Zhao in \cite{Z6}, inspired by a conjecture of Olivier Mathieu (\cite{Mat}), which was shown by Mathieu to imply the famed Jacobian Conjecture.   The third author then formulated the Image Conjecture (Conjecture \ref{ImConj}) upon noticing the resemblance of Mathieu's conjecture with his own Vanishing Conjecture, which he had shown to be equivalent to the Jacobian Conjecture (\cite{Z7}).  He proved that the Image Conjecture, for characteristic zero, implies the Vanishing Conjecture.  This connection makes the Image Conjecture a matter of intrigue.  The reader is referred to \cite{Essen} for more details on this story.

We begin by defining a Mathieu subspace.  Let $k$ be a field and $A$ a commutative $k$-algebra.  Consider the following two conditions relating to a $k$-vector subspace $\mathcal M$ of $A$ and an element $f$ of $A$:

\begin{equation}\text{$f^m\in\mathcal M$ for all $m\ge1$,}\tag{M1} \end{equation}

\noindent and

\begin{equation}\text{ for any $g\in A$, we have $f^mg\in\mathcal M$ for $m\gg0$.}\tag{M2} \end{equation}
We will refer to these conditions by their labels (M1) and (M2) throughout this paper.

\begin{defi} A sub-$k$-vector space $\mathcal M$ of $A$ is called a {\it Mathieu subspace} if, for all $f\in A$, (M1) implies (M2).
\end{defi}

It is not difficult to verify that in the definition of Mathieu subspace the condition (M1) can be replaced by
\begin{equation}\text{$f^m\in\mathcal M$ for all $m\gg0$,}\tag{M1$'$}\end{equation}
and although (M1) appeared in the original definition of Mathieu subspace given in \cite{Z6}, the authors have of late been stating the definition using (M1$'$), for the purpose of comparison with the definition of an ideal.  A proof of the equivalence of the two definitions has been given in Proposition 2.1 of \cite{Z8}.

We list some basic facts about Mathieu subspaces, which we leave to the reader to verify:
\begin{enumerate}
\item $A$ and $\{0\}$ are Mathieu subspaces.
\item\label{contains1} If $\mathcal M$ is a Mathieu subspace and $1\in\mathcal M$, then $\mathcal M=A$.
\item Any ideal in $A$ is a Mathieu subspace.
\item The sum $\mathcal M +\mathcal N$ of two Mathieu subspaces is not necessarily a Mathieu subspace.  (Hint: Use basic facts 2 and 3.  Or, see Example 4.12  
in \cite{Z6}.)
\end{enumerate}

In the next section we will state the Image Conjecture, for which the notion of a Mathieu subspace is needed, and prove some special cases.  Before we proceed, one more definition is in order.

\begin{defi}\label{Ldef}  For any ring $A$ and variables $z_1,\ldots, z_n$, let $\cL:A[z_1,\ldots,z_n]\to A$ be the $A$-linear map defined by $\cL(z^i)=i!$ (meaning $\cL(z_1^{\ell_1}\cdots z_n^{\ell_n})=\ell_1!\cdots \ell_n!$).
\end{defi}

Many of the results surrounding the conjecture involve this curious map $\cL$, which will be at the heart of the Factorial Conjecture, introduced and discussed in Section \ref{FactConj}.

\section{The Image Conjecture} 

The Image Conjecture, formulated by the third author in \cite{Z5}\footnote{The formulation in \cite{Z5} assumes $A$ is a $\Q$-algebra; however it is more general in its assumption about $\mathcal D$.  See Conjecture 1.3 in \cite{Z5}.}, goes as follows:
\begin{conj}[Image Conjecture]\label{ImConj} Let $k$ be a field and $A$ be a $k$-algebra, and let $B=A[z_1,\ldots,z_n]$ be the polynomial ring in $n$ variables over $A$.  For $a_1,\ldots,a_n\in A$ a regular sequence, the image of the $A$-linear map $B^n\to B$ defined by $\mathcal D=(\partial_{z_1}-a_1,\ldots,\partial_{z_n}-a_n)$ is a Mathieu subspace in $B$.
\end{conj}

We will begin by showing the Image Conjecture is true when $k$ has positive characteristic.  We are most interested, though, in the case when $k$ has characteristic zero, from which the Jacobian Conjecture would follow.  For the characteristic zero case we have only a partial result for $n=1$ (Theorem \ref{onevar} below); beyond that the Image Conjecture remains a mystery.

\begin{theo}\label{charp} Let $A$ be an $\mathbb F_p$-algebra, and let $B=A[z_1,\ldots,z_n]$ be the polynomial ring in $n$ variables over $A$.  For $a_1,\ldots,a_n\in A$ a regular sequence, the image of the $A$-linear map $B^n\to B$ defined by $\mathcal D=(\partial_{z_1}-a_1,\ldots,\partial_{z_n}-a_n)$ is a Mathieu subspace in $B$.
\end{theo}

\begin{rema}The theorem fails if we drop the hypothesis that $a_1,\ldots,a_n$ forms a regular sequence.  This can be seen in the case $n=1$, $A=\mathbb F_p$ (or any field of characteristic $p$), and $a_1=0$. In that case  $1=\partial_zz\in\text{Im}\,\mathcal D$, but $z^{p-1}\notin\text{Im}\,\mathcal D$, so $\text{Im}\,\mathcal D$ is not a Mathieu subspace by item \ref{contains1} in the Introduction.  (This is Example 2.7 in \cite{Z5}).  \end{rema}

Before proving Theorem \ref{charp} we need some preliminary results, the first of which is a well-known fact about regular sequences.

\begin{lemma}\label{koszul} Let $A$ be a ring and let $a_1,\ldots,a_n$ be a regular sequence $A$.  If $g_1,\ldots,g_n\in A$ are such that $\sum_{i=0}^{n}a_i g_i=0$, then for each pair $(i,j)$ with $1\le i,j\le n$ and $i\ne j$ there exists an element $g_{ij}\in A$ such that $g_{ij}=-g_{ji}$ for each pair and $g_i=\sum_{j\ne i}g_{ij}a_j$.
\end{lemma}

\begin{proof} This follows from the exactness of the Koszul complex for the sequence $(a_1,\ldots,a_n)$ (see \cite{Mats}, \S18.D).
\end{proof}

For the rest of this section $A$, $B$, $a_1,\ldots,a_n$, and $\mathcal D$ will be as in Theorem \ref{charp}, and $\mathfrak a$ will denote the ideal $Aa_1+\cdots+Aa_n$ of $A$.  We will write $z^r$ for the monomial $z_1^{r_1}\cdots z_n^{r_n}$.  For the very next result $A$ does not need to be an $\mathbb F_p$-algebra.

\begin{lemma}\label{crucial}  Let $g\in B=A[z]$ be of degree $d$, with $g_d$ its degree $d$ homogeneous summand.  If $g\in\mathrm{Im}\,\cD$, then all coefficients of $g_d$ belong to the ideal $\mathfrak a$.
\end{lemma}

\begin{proof} Being in the image of $\mathcal D$, $g$ has the form 
\begin{equation}\label{im}
g=\sum_{i=1}^n(\partial_{z_i}-a_i)h_i
\end{equation} 
for some $h_1,\ldots,h_n\in B$.  For $1\le i\le n$ and any integer $m\ge0$ we will denote by $h_{i,m}$ the degree $m$ homogeneous summand of $h_i$.  Let $e$ be the maximum of the degrees of $h_1,\ldots,h_n$. Since $\text{deg}\, g=d$, it is clear from (\ref{im}) that not all of $h_1,\ldots,h_n$ can have degree strictly less than $d$, so we have $e\ge d$.  If $e=d$ it follows from (\ref{im}) that $g_d=-\sum_{i=1}^na_ih_{i,d}$, and hence that all its coefficients belong to $\mathfrak a$, and we are done.  

If $e>d$ then it follows from (\ref{im}) that $\sum_{i=1}^na_ih_{i,e}=0$.  We appeal to Lemma \ref{koszul}, replacing $A$ with $B$ (which is innocent, since $a_1,\ldots,a_n$ is a regular sequence in $B$ as well), which asserts the existence of polynomials $p_{ij,e}\in B$, for $i\ne j$, such that $p_{ij,e}=-p_{ji,e}$ and $h_{i,e}=\sum_{j\ne i}p_{ij,e}a_j$.  Since each $h_{i,e}$ is homogeneous of degree $e$, we can replace $p_{ij,e}$ by its degree $e$ homogeneous summand and assume $p_{ij,e}$ homogeneous of degree $e$ as well.

More generally, we claim that for $m\ge d+1$ we have, for each pair $i,j$ with $i\ne j$, a polynomial $p_{ij,m}$, homogeneous of degree $m$ and $0$ if $m>e$, such that $p_{ij,m}=-p_{ji,m}$ and
\begin{equation}\label{p}
h_{i,m}=\sum_{j\ne i}(p_{ij,m}a_j-\partial_{z_j}p_{ij,m+1})\,.
\end{equation}
Note that the preceding paragraph established exactly this for $m=e$, with $p_{ij,e+1}=0$ as required.  Suppose inductively that the polynomials have been found for larger values of $m$.  Reading equation (\ref{im}) in degree $m$ gives 
{\allowdisplaybreaks
\begin{align}0&=\sum_{i=1}^n(\partial_{z_i}h_{i,m+1}-a_ih_{i,m})\label{first}\\
&=\sum_{i=1}^n\left(\partial_{z_i}\left(\sum_{j\ne i}(p_{ij,m+1}a_j-\partial_{z_j}p_{ij,m+2})\right)-a_ih_{i,m}\right)\notag\\
&=\sum_{i=1}^n\left(\partial_{z_i}\left(\sum_{j\ne i}p_{ij,m+1}a_j\right)-a_ih_{i,m}\right)-\sum_{i\ne j}\partial_{z_i}\partial_{z_j}p_{ij,m+2}\notag\\
&=\sum_{i=1}^n\left(\partial_{z_i}\left(\sum_{j\ne i}p_{ij,m+1}a_j\right)-a_ih_{i,m}\right)\quad\text{(since $\partial_{z_i}\partial_{z_j}p_{ij,m+2}=-\partial_{z_j}\partial_{z_i}p_{ji,m+2}$)}\notag\\
&=-\sum_{i=1}^na_i\left(h_{i,m}+\sum_{j\ne i}\partial_{z_j}p_{ij,m+1}\right)\quad\text{(this uses $p_{ij,m+1}=-p_{ji,m+1}$)}\,.\label{last}
\end{align}
}
From this equation, Lemma \ref{koszul} provides polynomials $p_{ij,m}\in B$ with $p_{ij,m}=-p_{ji,m}$  such that $h_{i,m}+\sum_{j\ne i}\partial_{z_j}p_{ij,m+1}=\sum_{j\ne i}p_{ij,m}a_j$, which, solving for $h_{i,m}$, yields (\ref{p}).

Finally, we complete the proof by reading (\ref{im}) in degree $d$, which gives $g_d$ as the right side of (\ref{first}) with $m=d$, and hence (following the same reasoning) $g_d$ is equal to (\ref{last}), with $m=d$.  This shows the coefficients of $g_d$ lie in $\mathfrak a$.
\end{proof}

We now will need to assume that $A$ is an $\mathbb F_p$-algebra.

\begin{corol}\label{cp} Let $f=\sum c_rz^r\in B$ with $c_r\in A$.  If $f^p\in\mathrm{Im}\,\cD$, then $c_r^p\in\mathfrak a$ for all $r$.
\end{corol}

\begin{proof}  The proof will be by induction on the number $d$ of non-zero homogeneous summands of $f$. Write $f=f_1+\cdots+f_d$ where $f_i$ are non-zero homogeneous summands with $\text{deg}\,f_i<\text{deg}\,f_j$ when $i<j$.  Then $f^p=f_1^p+\cdots+f_d^p$, and since $f^p\in\mathrm{Im}\,\cD$ Lemma \ref{crucial} says that all coefficients of $f_d^p$ belong to $\mathfrak a$, and this proves the case $d=1$.  In any case $f_d^p$ is the sum of monomials of the form $ca_iz^{pr}$ with $c\in A$, $r=(r_1,\ldots,r_n)$, $r_1+\cdots+r_d=\text{deg}\,f_d$.  Since $ca_iz^{pr}=(\partial_i-a_i)(-cz^{rp})\in\text{Im}\,\cD$, it follows that $f_d^p\in\text{Im}\,\cD$, so $f^p-f_d^p=f_1^p+\cdots+f_{d-1}^p\in\text{Im}\,\cD$, and the proof is complete by induction.
\end{proof}

\begin{lemma}\label{firstlemma} For all $r=(r_1,\ldots,r_n)$ we have $a_i^pz^r\in\mathrm{Im}\,\cD$.
\end{lemma}

\begin{proof} Since $\partial_i^p=0$ on $B$, we have $(-a_i)^pz^r=(\partial_i-a_i)^pz^r\in\mathrm{Im}\,\cD$.
\end{proof}
With these facts the proof of Theorem \ref{charp} follows quickly.

\begin{proof}[Proof of Theorem \ref{charp}]  We will show, more strongly, that if $f\in B$ with $f^p\in \text{Im}\,\cD$, then for any $g\in B$ we have $f^mg\in\text{Im}\,\cD$ when $m\ge p^2$.  Let $f=\sum c_rz^r$ be such that $f^p\in \text{Im}\,\cD$.  By Corollary \ref{cp} we have $c_r^p\in\mathfrak a$, hence $c_r^{p^2}\in Aa_1^p+\cdots+Aa_n^p$, for all $r$.  Since $f^{p^2}=\sum c_r^{p^2}z^{p^2r}$, it follows that for every $g\in B$ all coefficients of $f^mg$ belong to $Aa_1^p+\cdots+Aa_n^p$ if $m\ge p^2$.  Therefore $f^mg\in\text{Im}\,\cD$ by Lemma \ref{firstlemma}.
\end{proof}

For characteristic zero, the Image Conjecture is not even completely solved in the case $n=1$.  However, the theorem below solves a weak version of this case.  Here $z$ represents only one variable.

\begin{theo}\label{onevar}
If $A$ is a  $\Q$-algebra and if $a\in A$ is a non-zero-divisor such that $Aa$ is a radical ideal, then the image of $\mathcal D=\partial_z-a$ is a Mathieu subspace in $B=A[z]$.
\end{theo}

\begin{rema}\label{prerem}  The proof of this theorem will appeal to a result from Section \ref{FactConj}, namely Theorem \ref{n=1}, which says that if $f\in\C[z]$ ($z$ representing one variable) and $\cL(f^m)=0$ for all $m\ge0$, then $f=0$.  An easy use of the Lefschetz principle shows that the same holds replacing $\C$ by an arbitrary field of characteristic zero.
\end{rema}

In the case where $a$ is a unit in $A$ it can be shown rather easily that $\text{Im}\,\mathcal D=B$, hence is a Mathieu subspace.  Just note that $\partial_z-a$ has the inverse map $(\partial_z-a)^{-1}=[-a(1-a^{-1}\partial_z)]^{-1}=-a^{-1}\sum_{i=0}^{\infty}a^{-i}\partial_z^i$, which makes sense because $\partial_z$ is locally nilpotent. 

Therefore we make some preparations in the case $a$ is not a unit, in which case $I=\cap_{i=1}^\infty Aa^i\ne A$.  For $c\in A-I$ there exists a unique integer $m\ge0$ such that $c\in Aa^m-Aa^{m+1}$.  Setting $m=\infty$ if $c\in I$, we call $m$ the $a$-order of $c$ and denote it by $v_a(c)$.  Since $a$ is a non-unit in $B$ as well, $v_a$ extends to elements of $B$ which do not lie in $\cap_{i=1}^\infty Ba^i$. It is clear that an element $f$ of $B$ of the form $cz^i$, then $v_a(f)=v_a(c)$.  

In the following proposition $\mathcal D$ is as in Theorem \ref{onevar}.  Here $A$ can be any commutative ring, not necessarily a $\Q$-algebra.

 \begin{propo}\label{Lprop} Let $a\in A$ be a non-zero-divisor.    Let $f=b_0 +b_1z+\cdots+b_dz^d\in A[z]$.
\begin{enumerate}
\item[i)] If $f\in\text{\rm Im}\,\mathcal D$, then $b_d\equiv0\mod a$ and
\begin{equation}\label{apowers}
d!b_d+(d-1)!b_{d-1}a+(d-2)!b_{d-2}a^2+\cdots+b_0a^d\equiv0\mod a^{d+1}\,.
\end{equation}
\item[ii)] Conversely, let $A$ be either a $\Q$-algebra or an $\mathbb F_p$-algebra such that $d<p$.  If $f$ satisfies {\rm(\ref{apowers})}, then $f\in\text{\rm Im}\,\mathcal D$.
\end{enumerate}
 \end{propo}
 \begin{proof} For {\it i) }we can assume $b_d\ne0$.  If $d=0$ the two statements coincide and are easy to prove.  Assume $d\ge1$ and $g\in\text{Im}\,\mathcal D$, so that $f=(\partial_z-a)(c_0+c_1z+\cdots+c_dz^d)$. (Note that the polynomial on the inside must have the same degree as that of $f$, since $a$ is not a zero-divisor.)  In particular $b_d=-ac_d$, establishing the first assertion of {\it i), }and therefore $f-(\partial_z-a)(c_dz^d)=b_0+\cdots+b_{d-2}z^{d-2}+(b_{d-1}-dc_d)z^{d-1}\in\text{Im}\,\mathcal D$.  By induction on $d$ we have $(d-1)!(b_{d-1}-dc_d)+(d-2)!b_{d-2}a+\cdots+b_0a^{d-1}\equiv0 \mod a^d$.  Multiplying by $a$ and using $b_d=-ac_d$ gives (\ref{apowers}).
 
 For {\it ii), }note that the hypothesis and (\ref{apowers}) imply that $b_d=-ac_d$ for some $c_d\in A$.  If $d=0$ all is clear.  If $d\ge1$ we again have $f-(\partial_z-a)(c_dz^d)=b_0+\cdots+b_{d-2}z^{d-2}+(b_{d-1}-dc_d)z^{d-1}$, so $f\in\text{Im}\,\mathcal D$ if and only if $b_0+\cdots+b_{d-2}z^{d-2}+(b_{d-1}-dc_d)z^{d-1}\in\text{Im}\,\mathcal D$.  By induction it suffices to show $(d-1)!(b_{d-1}-dc_d)+(d-2)!b_{d-2}a+\cdots+b_0a^{d-1}\equiv0 \mod a^d$, or equivalently (since $a$ is a non-zero-divisor), that $(d-1)!(b_{d-1}a-dac_d)+(d-2)!b_{d-2}a^2+\cdots+b_0a^d\equiv0 \mod a^{d+1}$.  Since $ac_d=-b_d$, this is precisely the hypothesis.
 \end{proof}
 
 Now we return to our assumption that $A$ is a $\Q$-algebra.

\begin{lemma}\label{monom} An element of $B$ of the form $cz^i$ lies in the image of $\mathcal D$ if and only if $v_a(c)\ge i+1$.
\end{lemma}
\begin{proof}  This is immediate from Proposition \ref{Lprop}.
\end{proof}

\begin{corol}\label{goodf} Let $f=c_0+c_1z+\cdots+c_dz^d\in B$.  If $v_a(c_i)\ge i+1$ for $0\le i\le d$, then for each $g\in B$ we have $gf^m\in\text{\rm{Im}}\,\mathcal D$ for $m\gg0$.
\end{corol}

\begin{proof} Let $N=\text{deg}\,g$ and let $m\ge N+1$.  Note that each term $cz^j$ in $f^m$ satisfies $v_a(c)\ge j+m$.  Hence each term $cz^j$ of $gf^m$ satisfies $v_a(c)\ge j+m-N\ge j+1$.  By Lemma \ref{monom} each term of $gf^m$, and hence $gf^m$ itself, lies in $\text{\rm{Im}}\,\mathcal D$.
\end{proof}

\begin{lemma}\label{delete} Let $f=c_0+c_1z+\cdots+c_dz^d\in B$ be such that $v_a(c_i)\ge i$ for $0\le i\le d$, and, for some $t\le d$, $v_a(c_t)\ge t+1$.  Let $\tilde f=f-c_tz^t$.  If $f^m\in\text{\rm{Im}}\,\mathcal D$ for some $m\ge 1$, then $\tilde f^m\in\text{\rm{Im}}\,\mathcal D$.
\end{lemma}
 \begin{proof} Writing $\tilde f^m=f^m+h$ one easily sees that the terms of $h$ satisfy the hypothesis of Lemma \ref{monom}, and so we have $h\in\text{\rm{Im}}\,\mathcal D$.  Since $f^m\in\text{\rm{Im}}\,\mathcal D$, it follows that $\tilde f^m\in\text{\rm{Im}}\,\mathcal D$.
 \end{proof}

\begin{proof}[Proof of Theorem \ref{onevar}] Let $f=c_0+c_1z+\cdots+c_dz^d\in B$ be such that $f^m\in\text{\rm{Im}}\,\mathcal D$ for all $m\ge1$.  We will show that $v_a(c_i)\ge i+1$ for $0\le i\le d$, which implies $\text{\rm{Im}}\,\mathcal D$ is a Mathieu subspace by virtue of Corollary \ref{goodf}.

Suppose, to the contrary, that $v_a(c_i)\le i$ for some $i$.  Let $t$ be the maximum of the numbers $i-v_a(c_i)$, which, by our assumption is non-negative.  Let $h=a^tf$.  Then for each term $cz^i$ of $h$ we have $v_a(c_i)\ge i$, and equality holds for at least one $i$.  Clearly $h^m\in\text{\rm{Im}}\,\mathcal D$ for all $m\ge1$.  Using Lemma \ref{delete} to remove the terms for which equality does not hold, we arrive at a polynomial $f=c_0+c_1z+\cdots+c_dz^d\in B$ with $f^m\in\text{\rm{Im}}\,\mathcal D$ for all $m\ge1$ having the property that $v_a(c_i)=i$ when $c_i\ne0$.  We have $c_i=a^ib_i$ with $b_i\in A$, and when $b_i\ne0$ we have $b_i\notin Aa$.  Letting $p=\sum b_iz^i$ we then have $f=p(az)$.  

For any $g(z)\in B$, if $g$ has degree $\le N$ for some integer $N\ge0$, it follows from Proposition \ref{Lprop} that $g(az)\in\text{Im}\,\mathcal D$ if and only if $a^N\cL(g)\equiv0\mod a^{N+1}$ ($\cL$ as in Definition \ref{Ldef}).  Noting that $f^m=p^m(az)$ and $\text{deg}\,p^m\le md$ we thereby conclude $a^{md}\cL(p^m)\equiv0\mod a^{md+1}$ for all $m\ge1$.  Since $a$ is not a zero-divisor, we get $\cL(p^m)\equiv0\mod a$ for all $m\ge1$.  Let $s$ be the smallest of all $i$ such that $b_i\ne0$.  Then $b_s\notin Aa$.  We are assuming $Aa$ is a radical ideal, hence it is the intersection of the prime ideals containing it.  Therefore there is a prime ideal $\mathcal P$ in $A$ containing $Aa$ but not containing $b_s$.  Letting $\bar p$ be the image of $p$ in $k[z]$ where $k$ is the fraction field of $A/\mathcal P$, we have $\bar p\ne0$ and $\cL(\bar p^m)=0$.  But this contradicts Theorem \ref{n=1} (see Remark \ref{prerem}).
\end{proof}

\section{Specific version of the Image Conjecture relevant to the Vanishing and Jacobian Conjectures} 

The following specific version of the Image Conjecture, from \cite{Z5}, is of special interest.   For this we let $\xi =(\xi_1,\ldots,\xi_n)$ and $z=(z_1,\ldots,z_n)$ be two sets of commuting indeterminates,  and we consider the commuting operators $\mathcal D_i=\xi_i-\partial_{z_i}$, $1\le i\le n$, on the polynomial ring $A=\C[\xi,z]$.  We consider the map $\mathcal D=(\mathcal D_1,\ldots,\mathcal D_n):A^n\to A$.  

\begin{conj}[Special Image Conjecture]\label{image}  The image of $\mathcal D$ is a Mathieu subspace.
\end{conj}

In \cite{Z5} it is shown that the above conjecture implies the Jacobian Conjecture.\footnote{One has to prove the conjecture for all $n\ge1$, which then implies the Jacobian Conjecture for all $n\ge1$.}  More specifically, it is shown that it suffices to show that 

\begin{theo}[\cite{Z5}, Theorem 3.7]  The following two statements are equivalent:
\begin{enumerate}
\item For any $f\in\C[\xi,z]$ of the form $(\xi_1^2+\cdots+\xi_n^2)P$ with $P\in\C[z]$ and $P$ is homogeneous of degree four, then $f^m\in\mathrm{Im}\,\mathcal D$ for all $m\ge1$ implies that, for each $g\in\C[z]$, $f^mg\in\mathrm{Im}\,\mathcal D$ for all $m\gg0$.
\item The Jacobian Conjecture holds in all dimensions $n\ge1$.
\end{enumerate}
\end{theo}

We now give a realization of the image of $\mathcal D$ that is established in \cite{Z5}.  Let $\mathcal E$ be the $\C$-linear map from $\C[\xi,z]$ to $\C[z]$ defined by sending a monomial $\xi_1^{\alpha_1}\cdots\xi_n^{\alpha_n}z_1^{\beta_1}\cdots z_n^{\beta_n}$ to $\partial_{z_1}^{\alpha_1}\cdots\partial_{z_n}^{\alpha_n}z_1^{\beta_1}\cdots z_n^{\beta_n}$.  Then:

\begin{theo}[\cite{Z5}, Theorem 3.1]\label{im=ker} $\mathrm{Im}\,\mathcal D=\mathrm{Ker}\,\mathcal E$.
\end{theo}
\noindent This obviously makes it much easier to determine whether an element lies in $\text{Im}\,\mathcal D$, as $\mathcal E$ is easy to apply.

We now set $\mathcal M=\text{Im}\,\mathcal D\,(=\text{Ker}\,\mathcal E)$ and make a number of observations, letting $A=\C[\xi,z]$ as above, first noting that, by Theorem \ref{im=ker}, condition (M1) coincides with
\begin{equation}\text{$\mathcal E(f^m)=0$ for all $m\ge1$}\notag \end{equation}
in this context.

We define a multi-grading on the polynomial ring $\C[\xi,z]$ by setting the multi-degree of a monomial $\xi_1^{i_1}\cdots\xi_n^{i_n}z_1^{j_1}\cdots z_n^{j_n}$ to be $(j_1-i_1,\ldots,j_n-i_n)$.  We also have the ordinary grading on $\C[\xi,z]$ by which $\xi_1,\ldots,\xi_n$ each have degree $-1$ and $z_1,\ldots,z_n$ each have degree $1$.  The motivation for these choices is the map $\mathcal E$, which preserves $z_1,\ldots,z_n$ but converts $\xi_1,\ldots,\xi_n$ to operators which lower degree by one.  In the discussion below, ``multi-degree" refers to the former; ``degree" refers to the latter.  With $\C[z]$ viewed as a subring of $A=\C[\xi,z]$, these gradings restrict to give a multi-grading and a grading on $\C[z]$.  Note that the map $\mathcal E:A\to\C[z]$ preserves both the multi-degree and the degree of a monomial.

\begin{enumerate}

\item  Condition (M2) is satisfied if it holds whenever $g$ is a monomial in $A$.

\item We can write any $f\in A$ as a sum of terms of the form $z_1^{r_1}\cdots z_n^{r_n}Q$ where $Q$ has multi-degree $(0,\ldots,0)$, and $(r_1,\ldots,r_n)\in\mathbb Z^n$.  These terms are just the multi-homogeneous summands of $f$.  Any $Q(\xi,z)$ of multi-degree $(0,\ldots,0)$ can be written in the form $q(U_1,\ldots,U_n)$ where $U_i=\xi_iz_i$ for $i=1,\ldots,n$.  

\item If $f$ is multi-homogeneous of multi-degree $(r_1,\ldots,r_n)$, in other words if $f$ has the form $z_1^{r_1}\cdots z_n^{r_n}q(U_1,\ldots,U_n)$, then:

\begin{enumerate}
\item If $r_1,\ldots,r_n\ge0$ then $\mathcal E(f)=cz_1^{r_1}\cdots z_n^{r_n}$ for some $c\in\C$ (since $\mathcal E$ preserves multi-degree).
\item If $r_i<0$ for some $i$ then $\mathcal E(f)=0$.
\end{enumerate}

Note that if (b) holds for $f$ then it holds for $f^m$ for any $m\ge1$, hence (M1) holds for $f$.  Moreover it's easy to see that, for any $g\in A$, (b) holds for all multi-homogeneous terms of $f^mg$, for $m\gg0$, so (M2) holds for $f$ as well.

\item  For any $f\in A$, let $N_f$ be the convex polyhedron (Newton polyhedron) in $\mathbb R^n$ determined by the finite set of points $(r_1,\ldots,r_n)$ which are multi-degrees of the nonzero terms $z_1^{r_1}\cdots z_n^{r_n}q(U)$ (as above) appearing in $f$.  

\item  Note that if $f\in A$ is such that there exists $i$ such that the multi-degree of all multi-homogeneous summands of $f$ have negative $i$-coordinate, then again we have $\mathcal E(f^m)=0$ for all $m\ge1$ and $\mathcal E(f^mg)=0$ for all $g\in A, m\gg0$, hence $f$ satisfies (M1) and (M2).  This condition simply says that $N_f$ lies in the half space $\{(x_1,\ldots,x_n)\in\mathbb R^n\,|\,x_i<0\}$.

\item More generally, if there exists a hyperplane $H\subset\mathbb R^n$ through the origin such that the strictly positive $n$-tant $\{(x_1,\ldots,x_n)\in\mathbb R^n\,|\,x_1,\ldots,x_n>0\}$ and $N_f$ lie strictly on opposite sides of $H$, then $\mathcal E(f^m)=0$ for all $m\ge1$ and $\mathcal L(f^mg)=0$ for all $g\in A, m\gg0$, hence $f$ satisfies (M1) and (M2).  This can be seen as follows:  There is a nonzero vector $v=(v_1,\ldots,v_n)\in\mathbb R^n$ such that $v_1,\dots,v_n\ge0$ and such that $H=\{x\in\mathbb R^n\,|\,(x\cdot v)=0\}$ (usual inner product).  Then $(v\cdot r)<0$ for all $r\in N_f$.  It follows that for all terms $z_1^{s_1}\cdots z_n^{s_n}q(U_1,\ldots,U_n)$ of $f^m$, where $m\ge1$, we must have $(v\cdot s)<0$, where $s=(s_1,\ldots,s_n)$ (in other words all points on the Newton polyhedron of $f^m$ lies below $H$). Therefore we must have $s_i<0$ for some $i$, from which it follows that $\mathcal E(f^m)=0$.  Similarly, if $g\in A$ then for sufficiently large $m$, all points in the Newton polyhedron of $f^mg$ are below $H$, so that $\mathcal E(f^mg)=0$.

\item If $f\in A$ and $N_f$ has an extremal point $(r_1,\ldots,r_n)$ corresponding to the term $z^rq(U)=z_1^{r_1}\cdots z_n^{r_n}q(U_1,\ldots,U_n)$, then the point $(mr_1,\ldots,mr_n)$ lies on the Newton polyhedron of $f^m$ (from the term $z^{mr}q(U)^m=z_1^{mr_1}\cdots z_n^{mr_n}q(U_1,\ldots,U_n)^m$), and in fact is an extremal point.  Thus if $f$ satisfies (M1), so does the multi-homogeneous summand $z^rq(U)$.

\item\label{point1} We suspect that it cannot happen that a nonzero multi-homogeneous element $z^rq(U)$ with $r_1,\ldots,r_n\ge0$ satisfies (M1).  If this suspicion is true, then by the last item, the Newton polyhedron of an $f\in A$ satisfying (M1) cannot have an extremal point in the closed positive $n$-tant $\{(x_1,\ldots,x_n)\in\mathbb R^n\,|\,x_1,\ldots,x_n\ge0\}$.

\item\label{point2} To address the problem in the previous item, note that if a multi-homogeneous element $f=z_1^{r_1}\cdots z_n^{r_n}q(U_1,\ldots,U_n)$ satisfies (M1), i.e., $\mathcal E(f^m)=0$ for all $m\ge1$, then so does $\xi_1^{r_1}\cdots\xi_n^{r_n}f=U_1^{r_1}\cdots U_n^{r_n}q(U)$, which has multi-degree $(0,\ldots,0)$.  Thus we need to show that if $h\in \C[U_1,\ldots,U_n]$ and if $\mathcal E(h^m)=0$ for all $m\ge1$, then $h=0$.  This will be Conjecture \ref{numbers} below.

\end{enumerate}

Recall that $U_i=\xi_iz_i$.  One sees that for a monomial $U^\ell=U_1^{\ell_1}\cdots U_n^{\ell_n}$ we have $\mathcal E(U^\ell)=\ell!=\ell_1!\cdots \ell_n!$.   Thus the map $\mathcal E$ restricted to $\C[U_1,\ldots,U_n]$ is precisely the map $\cL$ of Definition \ref{Ldef}. In the conjectures below $U=(U_1,\ldots,U_n)$ can be taken to be any system of variables (forgetting $\xi$ and $z$ for the moment), and $\cL:\C[U_1,\ldots,U_n]\to\C$ the $\C$-linear map sending $U^\ell$ to $\ell!$.

\section{The Factorial Conjecture}\label{FactConj}

It follows from the discussion of the preceding section that the following assertion, which draws interest merely by virtue of its simplicity, is necessary for the Image Conjecture to hold.

\begin{conj}\label{weaknumbers} The kernel of $\cL:\C[U_1,\ldots,U_n]\to\C$ is a Mathieu subspace.
\end{conj}

As per items \ref{point1} and \ref{point2} above, we propose the stronger assertion, which we dub the Factorial Conjecture:

\begin{conj}[Factorial Conjecture]\label{numbers} Let $f\in\C[U_1,\ldots,U_n]$ be such that $\cL(f^m)=0$ for all $m\ge1$.  Then $f=0$.
\end{conj}

As seen above, this conjecture would imply that the Newton polyhedron of any $f\in A=\C[\xi,z]$ satisfying (M1) has no extremal points in the closed positive $n$-tant.  

The Factorial Conjecture looks innocent on first glance; one would think it is either easy to prove or else a counterexample should be findable.  However no proof or counterexample has yet been given.  The authors believe it to be true and will devote quite a bit of effort below in showing that the condition $\cL(f^m)=0$ for all $m\ge1$ implies $f=0$ in various situations.  In this case we say ``the Factorial Conjecture holds for $f$".

As a first observation, let us note that the Factorial Conjecture holds for $f=cM$ where $c\in\C$ and $M$ is a monomial in $\C[U]$, since the condition $\cL(f)=0$ obviously implies $c=0$.  More strongly we have:

\begin{propo}\label{twomonom}  The Factorial Conjecture holds for $f\in\C[U_1,\ldots,U_n]$ of the form $c_1M_1+c_2M_2$, where $M_1,M_2$ are monomials and $c_1,c_2\in\C$.  More strongly, $\cL(f)=\cL(f^2)=0$ implies $f=0$ in this case.
\end{propo}

The proof will involve the following observation.

\begin{rema}\label{form}  The one-variable formula $\int_0^\infty U^k e^{-U}dU=k!$ (easily proved inductively using integration by parts) leads to the multi-variable formula
$$\int_{D_n}U^ke^{-U}dU=k!$$
where $U^k=U_1^{k_1}\cdots U_n^{k_n}$ and $k!=k_1!\cdots k_n!$, $dU=dU_1\cdots dU_n$, and $D_n$ is the non-negative $n$-tant $U_1\ge0,\ldots,U_n\ge0$ in $\R^n$.  It follows that for $f\in\C[U_1,\ldots,U_n]$, $\cL(f)$ can be realized as
\begin{equation}
\cL(f)=\int_{D_n}f(U)e^{-U}dU
\end{equation}
(which, incidentally, gives a way to calculate $\cL(f)$ using a symbolic algebra program such as Maple).  Letting $\langle\,\,,\,\,\rangle$ be the Hermitian inner product defined on $\C[U]$ by 
\begin{equation}
\langle f,g\rangle=\int_{D_n}f(U)\overline{g(U)}e^{-U}dU
\end{equation} 
we note that this restricts to a positive definite form on $\R[U]$, and that $\cL(f^2)=\langle f,f\rangle$, which must be strictly positive if $f\in\R[U]$ and $f\ne0$.
\end{rema}

\begin{proof}[Proof of Proposition \ref{twomonom}]  We have $\cL(f)=c_1L_1+c_2L_2=0$ with $L_1,L_2\in\Z-\{0\}$, so $c_2=-c_1L_1/L_2$ and $f=c_1h$ where $h=M_1-(L_1/L_2)M_2\in\Q[U]-\{0\}$.  From Remark \ref{form} we have $0=\cL(f^2)=\langle f,f\rangle=c_1\bar{c_1}\langle h,h\rangle$, which shows $c_1=0$, since $\langle h,h\rangle>0$.  By symmetry we have $c_2=0$, so $f=0$.
\end{proof}

Now we make two remarks that will be important in several of the proofs that follow.\footnote{It should be acknowledged that the technique of making reductions using these ideas is due to Mitya Boyarchenko.}  The first remark shows that to prove the Factorial Conjecture we may assume $f$ has coefficients which are algebraic numbers.  

\begin{rema}[Algebraic reduction]\label{algred} Given a collection of monomials $M_1,\ldots,M_d\in\C[U]$ (where $U$ represents $U_1,\ldots,U_n$), we consider whether there exists $f\ne0$ of the form $\sum_{i=1}^d c_iM_i$ which satisfy $\cL(f^m)=0$ for all $m\ge1$.  Thinking of of $c_1,\ldots,c_d$ as indeterminates, we note that $\cL(f^m)$ is a homogeneous polynomial of degree $m$ in $\Z[c_1,\ldots,c_d]$.  By the Nullstellensatz, the existence of a nonzero solution is equivalent to saying the polynomials $\cL(f^m)$ generate a homogeneous ideal in $\Q[c_1,\ldots,c_d]$ whose radical is strictly contained in the ideal generated by the indeterminates $c_1,\ldots,c_d$, which, in turn, is equivalent to the existence of a nonzero solution over $\overline\Q$, the algebraic closure of $\Q$. Similarly, if $f$ has the form $h+\sum_{i=1}^d c_iM_i$ where $h$ is a nonzero polynomial in $\Q[U]$ not involving the monomials $M_1,\ldots,M_d$, then consider the ideal generated by the (non-homogeneous) polynomials $\cL(f^m)$ in $\Q[c_1,\ldots,c_d]$.  The existence of a solution over $\C$ is equivalent to saying this ideal is not all of $\Q[c_1,\ldots,c_d]$, which is equivalent to the existence of a solution over $\overline\Q$.\end{rema}

\begin{rema}[Extension of primes]\label{valuation}Given any $c_1,\ldots,c_d\in\overline\Q$, the ring $\Q[c_1,\ldots,c_d]$ has a ring extension $\mathcal O$ in $\overline\Q$ which is integral over $\Z[1/\ell]$, for some $\ell\in\Z$, and we can take $\mathcal O$ to be a Dedekind ring (replacing $\mathcal O$ by its integral closure).  Hence for all but finitely many primes $p\in\Z$ (specifically, those primes not dividing $\ell$), $p\Z$ extends to a prime ideal of $\mathcal O$, or, equivalently, $\mathcal O$ has a (not necessarily unique) valuation $v_p$ which has positive value at $p$.  We will say ``$v_p$ is a valuation lying over $p$".  For $k\in\Z$ it will then be the case that $v_p(k)>0$ if and only if $p$ divides $k$ in $\Z$.
\end{rema}

\begin{propo}\label{monomtimes} The Factorial Conjecture holds for $f\in\C[U_1,\ldots,U_n]$ having the form $f=Mh$ where $M$ is a monomial and $h$ has nonzero constant term.
\end{propo}

\begin{proof} Suppose such an $f$ has the property $\cL(f^m)=0$ for $m\ge1$.  We can assume the constant term of $h$ is $1$, and that $h\ne1$  Then $f=M+c_1M_1+\cdots +c_dM_d$ where $M_1,\ldots,M_d$ are monomials properly divisible by $M$.  For any prime $p\in\Z$ we have
\begin{equation}\label{eq1}
f^p=M^p+\sum_{i=1}^d c_i^pM_i^p+p\sum_jg_j(c_1,\ldots,c_d)N_j
\end{equation}
where, for each $j$, $g_j(c_1,\ldots,c_d)\in\Z[c_1,\ldots,c_d]$ and $N_j$ is a monomial divisible by $M^p$.
Write $M=U^\alpha,M_1=U^{\alpha_1},\ldots,M_d=U^{\alpha_d}$, and $N_j=U^{\beta_j}$.  Applying $\cL$ to (\ref{eq1}) yields
\begin{equation}\label{eq2}
\cL(f^p)=(p\alpha)!+\sum_{i=1}^d c_i^p(p\alpha_i)! +p\sum_jg_j(c_1,\ldots,c_d)\beta_j!=0\,.
\end{equation}
We make two observations: Since $M$ properly divides $M_i$, we have $\alpha<\alpha_i$, so $(p\alpha)!$ divides $(p\alpha_i)!$ in $\Z$ and moreover, $p$ divides $(p\alpha_i)!/(p\alpha)!$ in $\Z$. Secondly, since $M^p$ divides $N_j$, $(p\alpha)!$ divides $(p\beta_j)!$ in $\Z$.  Dividing (\ref{eq2}) by $(p\alpha)!$, we get
\begin{equation*}
1+\sum_{i=1}^d c_i^p\frac{(p\alpha_i)!}{(p\alpha)!} +p\sum_jg_j(c_1,\ldots,c_d)\frac{\beta_j!}{(p\alpha)!}=0\,,
\end{equation*}
which shows that $p$ divides $1$ in $\Z[c_1,\ldots,c_d]$.  However, only finitely many primes can be units in $\Z[c_1,\ldots,c_d]$, so choosing $p$ to avoid this finite set brings us to a contradiction.
\end{proof}

Proposition \ref{monomtimes} has these two immediate consequences:

\begin{propo}\label{const} The Factorial Conjecture holds for $f\in\C[U_1,\ldots,U_n]$ having nonzero constant term.
\end{propo}

\begin{proof}
Apply Proposition \ref{monomtimes} with $M=1$.
\end{proof}

\begin{theo}\label{n=1} The Factorial Conjecture holds for $n=1$.
\end{theo}

\begin{proof}
Any nonzero polynomial in one variable has the form $f=Mh$ of Proposition \ref{monomtimes}.
\end{proof}

The following says something a little different from Proposition \ref{monomtimes}.

\begin{propo}\label{smallest} The Factorial Conjecture holds for $f\in\C[U_1,\ldots,U_n]$ of the form $cM_0+\sum_{i=1}^d c_iM_i$ where $M_0=U_1^{k_1}\cdots U_n^{k_n}$ with $k_1\ge1$ and $k_1\ge k_i$ for $i=2,\ldots,n$, $c,c_1,\ldots,c_d\in\C$ with $c\ne0$, and $M_1,\ldots,M_d$ are monomials each divisible by $U_1^{k_1+1}$.
\end{propo}

\begin{proof}  Assume such an $f$ has the property $\cL(f^m)=0$ for $m\ge1$.  We may assume $c=1$ and that $c_1,\ldots,c_d\in\overline\Q$, by Remark \ref{algred}.  Choose a Dedekind overring $\mathcal O$ of $\Z[c_1,\cdots,c_d]$ as in Remark \ref{valuation}.  Writing
\begin{equation*}
\begin{aligned}
f^m&=(M_0+\sum_{i=1}^d c_iM_i)^m\\
&=\sum_{i_0+i_1+\cdots+i_d=m}\binom{m}{i_0,i_1,\ldots,i_d}c_1^{i_1}\cdots c_d^{i_d}M_0^{i_0}M_1^{i_1}\cdots M_d^{i_d}\\
&=M_0^m+\sum_{i=1}^m\sum_{i_1+\cdots+i_d=i}\frac{m!}{(m-i)!i_1!\cdots i_d!}c_1^{i_1}\cdots c_d^{i_d}M_0^{m-i}M_1^{i_1}\cdots M_d^{i_d}\,,
\end{aligned} 
\end{equation*}
we have
\begin{equation}\label{m1}
0=\cL(f^m)=\cL(M_0^m)+\sum_{i=1}^m\sum_{i_1+\cdots+i_d=i}\frac{m!}{(m-i)!i_1!\cdots i_d!}c_1^{i_1}\cdots c_d^{i_d}\cL(M_0^{m-i}M_1^{i_1}\cdots M_d^{i_d})
\end{equation}
Let us note that, by our assumption about $M_0$, $mk_1+1$ does not divide $\cL(M_0^m)$ in $\Z$ if $mk_1+1$ is prime.  Also, by our assumptions about $M_1,\ldots,M_d$, $mk_1+1$ does divide each of the terms $\cL(M_0^{m-i}M_1^{i_1}\cdots M_d^{i_d})$ appearing in (\ref{m1}).  Using Dirichlet's prime number theorem\footnote{which asserts that for any two positive coprime integers $a$ and $b$, there are infinitely many primes of the form $a + nb$, where $n\ge0$. See Theorem 66 and Corollary 4.1 in \cite{F-T}.} we can select a prime number $p$ of the form $mk_1+1$ for which $\mathcal O$ has a valuation $v_p$ over $p$.  Viewing (\ref{m1}) as an equation in $\mathcal O$, we see that $v_p$ takes on positive values at each summand $\cL(M_0^{m-i}M_1^{i_1}\cdots M_d^{i_d})$.  For the first term, however, we have $\cL(M_0^m)=k_1!\cdots k_n!$, which is not divisible by $p$ in $\Z$ by our assumption, and hence $v_p(\cL(M_0^m))=0$.  This gives a contradiction, since the sum is 0.
\end{proof}

\begin{propo}\label{formpower} The Factorial Conjecture holds for $f\in\C[U_1,\ldots,U_n]$ a power of a linear homogenous form.  
\end{propo}

\begin{proof}  We have $f=g^r$ where $g=\sum_{i=1}^nc_iU_i$.  We concern ourselves with $g$ for a moment.   For $m>0$ an integer we have $g^m=\sum_{i_1+\cdots+i_n=m}\binom{m}{i_1,\ldots,i_n}c_1^{i_1}\cdots c_n^{i_n}U_1^{i_1}\cdots U_n^{i_n}$.  Thus $\cL(g^m)=\sum_{i_1+\cdots+i_n=m}\frac{m!}{i_1!\cdots i_m!}c_1^{i_1}\cdots c_n^{i_n}i_1!\cdots i_m!=m!\sum_{i_1+\cdots+i_n=m}c_1^{i_1}\cdots c_n^{i_n}$.  Let us denote by $h_m$ the polynomial $\sum_{i_1+\cdots+i_n=m}c_1^{i_1}\cdots c_n^{i_n}$, viewing $c_1,\ldots,c_n$ as indeterminates for the moment.  

The polynomials $h_1,h_2,\ldots\in\C[c_1,\ldots,c_n]$ are related to the elementary symmetric polynomials $s_1,\ldots,s_n$ (where $s_m=\sum_{1\leq i_1<\cdots<i_m\leq n}c_{i_1}\cdots c_{i_m}$) in the following way:  Let $T$ be an indeterminate, and set $S(T)=\prod_{i=1}^n(1-c_iT)=1-s_1T+s_2T^2-\cdots+(-1)^ns_nT^n$.  In $\C[c_1,\ldots,c_n][[T]]$ we have $S(T)^{-1}=\prod_{i=1}^n\frac{1}{(1-c_iT)}=\prod_{i=1}^n(1+c_iT+c_i^2T^2+\cdots)=1+h_1T+h_2T^2+\cdots$, and we let $P(T)$ be the latter power series.  Now we specialize to $c_1,\ldots,c_n\in\C$ and view $S(T)$ and $P(T)$ as elements of $\C[T]$, $\C[[T]]$, respectively.

Returning to $f=g^r$,  we see that our hypotheses $\cL(f^m)=0$ for $m\ge1$ says that $h_{mr}=0$ for $m\ge1$.  By Theorem \ref{gap}, we must have $S(T)=1$, i.e., $s_1,\ldots,s_n$ vanish at $(c_1,\ldots,c_n)$.  It is well-known (and easily seen) that the only zero of $s_1,\ldots,s_n$ is $(0,\ldots,0)$, so we must have $g=0$.
\end{proof}

\begin{rema} In the case where $f$ itself is a linear form one can easily see from the proof that, more strongly, $\cL(f)=\cL(f^2)=\cdots=\cL(f^n)=0$ implies $f=0$.
\end{rema}

\begin{theo}[N. Mohan Kumar]\label{gap}  Let $S(T)\in\C[T]$ with constant term $1$, and let $P(T)=1+a_1T+a_2T^2+\cdots$ be it's multiplicative inverse in the power series ring $\C[[T]]$. If there exists an integer $r>0$ such that $a_{mr}=0$ for all $m\ge1$, then $S(T)=1$.
\end{theo}
\begin{proof} We note that $\C[[T]]$ is a free module over $B=\C[[T^r]]$ with basis $\{1,T,\cdots,T^{r-1}\}$, and that $\C[T]$ is free over $A=\C[T^r]$ with the same basis.   Accordingly, we write $P(T)=B_0 +B_1T+\cdots+B_{r-1}T^{r-1}$ and $S(T)=A_0+A_1T+\cdots+A_{r-1}T^{r-1}$ with $B_0,\ldots,B_{r-1}\in B$, and $A_0,\ldots,A_{r-1}\in A$.  Our assumption about $P(T)$ clearly shows $B_0=1$, since the constant term is the only power of $T^r$ that has non-zero coefficient.  Now we tensor $\C[T]$ and $\C[[T]]$ with the rational function field $K=\C(T^r)$, which is the field of fractions of $A$.  This gives the containment $\C[T]\otimes_{A}K\subset\C[[T]]\otimes_{A}K$.  The first ring is the field $\C(T)$ (since $T$ is algebraic over $\C(T^r)$), which is free over $K=\C(T^r)$ with basis $\{1,T,\cdots,T^{r-1}\}$; the second ring is the field of Laurent power series ring $\C[[T]][T^{-1}]$, which is free with the same basis over $L=\C[[T^r]]\otimes_{A}K=\C[[T^r]][T^{-r}]$, which is the field of Laurent power series in $T^r$.  So we have:
$$\begin{matrix}
S(T)&=&A_0+A_1T+\cdots+A_{r-1}T^{r-1}&{}&1 +B_1T+\cdots+B_{r-1}T^{r-1}&=&P(T)\\
{}&{}&\begin{turn}{-90}\hskip-7pt$\in$\end{turn}&{}&\begin{turn}{-90}\hskip-7pt$\in$\end{turn}\\
{}&{}&A\oplus AT\oplus\cdots\oplus AT^{r-1}&\subset& B\oplus BT\oplus\cdots\oplus BT^{r-1}\\
{}&{}&\begin{turn}{-90}\hskip-7pt$\subset$\end{turn}&{}&\begin{turn}{-90}\hskip-7pt$\subset$\end{turn}\\
\C(T)&=&K\oplus KT\oplus\cdots\oplus KT^{r-1}&\subset& L\oplus LT\oplus\cdots\oplus LT^{r-1}
\end{matrix}$$
Since $S(T)$ lies in the field $\C(T)=K\oplus KT\oplus\cdots\oplus KT^{r-1}$, so must its inverse $P(T)$, and this shows that $B_1,\ldots,B_{r-1}$ lie in $K=\C(T^r)$.  Let $Q\in\C[T^r]$ be a common denominator for $B_1,\ldots,B_{r-1}$ as rational functions in $T^r$.  Then $$Q=QP(T)S(T)=(Q+QB_1T+\cdots+QB_{r-1}T^{r-1})S(T)\,.$$  Since $Q,QB_1,\ldots,QB_{r-1}$ all lie in $\C[T^r]$ there is no cancellation amongst summands of $Q+QB_1T+\cdots+QB_{r-1}T^{r-1}$.  Hence its degree is at least the degree of $Q$.  This shows the degree of $S(T)$ is zero, i.e., $S(T)=1$, as desired.
\end{proof}

We have not succeeded in proving that the Factorial Conjecture holds for more general homogeneous polynomials, except in a few situations given below.

\begin{propo} The Factorial Conjecture holds for $f\in\C[U_1,U_2]$ a quadratic homogenous form in two variables.
\end{propo}

\begin{proof} Writing $f=c_{20}U_1^2+c_{11}U_1U_2+c_{02}U_2^2$ we have
$$
f^m=\sum_{i+j+k=m}\frac{m!}{i!j!k!}c_{20}^ic_{11}^jc_{02}^kU_1^{2i+j}U_2^{j+2k}
$$
so that

\begin{align}\cL(f^m)&=\sum_{i+j+k=m}\frac{m!}{i!j!k!}c_{20}^ic_{11}^jc_{02}^k(2i+j)!(j+2k)!\notag\\
&=\sum_{0\le i+k\le m}\frac{m!}{i!(m-i-k)!k!}c_{20}^ic_{02}^kc_{11}^{m-i-k}(m+i-k)!(m-i+k)!=0\label{otherterms}\,.
\end{align}

Let $M$ be the sum of the terms in above where $k=i$, i.e.,
$$
M=\sum_{0\le2i\le m}\frac{m!}{(i!)^2(m-2i)!}(m!)^2(c_{20}c_{02})^ic_{11}^{m-2i}\,.
$$
By the integrality reduction (Remark \ref{algred}) we can assume $c_{20},c_{11},c_{02}$ lie in a ring $\mathcal O$ which is Dedekind and integral over $\Z[1/\ell]$ for some $\ell\in\Z$, $\ell\ne0$.  Let $p=2r+1\in\Z$ be an odd prime which corresponds to a valuation in $\mathcal O$, and consider the above equations with $m=2r$.  Let us note that $p$ divides all of the summands of (\ref{otherterms}) except those comprised by $M$, i.e., those for which $k\ne i$ (for if, say, $i>k$, then $p\,|\,(m+i-k)!\,$).  Thus we have $\cL(f^m)\equiv M \mod p$.  From Lemma \ref{congruence} below we get
$$
0\equiv M\equiv\sum_{i=0}^r\binom{r}{i}c_{11}^{2r-2i}(-4c_{20}c_{02})^i=(c_{11}^2-4c_{20}c_{02})^r\mod p
$$
Hence $p$ divides $(c_{11}^2-4c_{20}c_{02})^r$ in $\mathcal O$.   This shows that $d=c_{11}^2-4c_{20}c_{02}$ has a positive valuation for infinitely many valuations of $\mathcal O$, which shows that $d=0$.  Since $d$ is the discriminant of  $f$, we conclude that $f$ is the square of a linear form in $\C[U_1,U_2]$, so we are in the situation of Proposition \ref{formpower}, and the proof is complete, modulo the lemma below.
\end{proof}

\begin{lemma}\label{congruence} For $p=2r+1\in \Z$ an odd prime, we have, setting $m=2r$,
$$\frac{(m!)^3}{(i!)^2(m-2i)!}\equiv\binom{r}{i}(-4)^i\mod p$$
for $0\le i\le r$.
\end{lemma}

\begin{proof}  We have $m!\equiv-1\mod p$ by Wilson's Theorem,\footnote{Wilson's Theorem: An integer $n>1$ is prime if and only if $(n-1)!\equiv-1\mod n$.  See \cite{Wik} for a very nice survey on this.} so it remains to prove that
\begin{equation}\label{congruence2}
i!(2r-2i)!\frac{r!}{(r-i)!}(-4)^i\equiv-1\mod p\,.
\end{equation}
To see this, we begin with the expression on the left:
\begin{align*}
i!(2r-2i)!\frac{r!}{(r-i)!}(-4)^i&=i!(2r-2i)!r(r-1)\cdots(r-i+1)2^i(-2)^i\\
&=i!(2r-2i)!2r(2r-2)\cdots(2r-2i+2)(-2)^i\\
&=\frac{i!(2r)!}{(2r-1)(2r-3)\cdots(2r-2i+1)}(-2)^i\\
&=\frac{i!(p-1)!}{(p-2)(p-4)\cdots(p-2i)}(-2)^i\\
&\equiv\frac{i!(-1)}{(-2)(-4)\cdots(-2i)}(-2)^i\\
\intertext{(going mod $p$ and again appealing to Wilson's Theorem)}
&\equiv\frac{i!(-1)}{(i!)(-2)^i}(-2)^i\equiv-1\,.
\end{align*}
\end{proof}

\begin{propo}\label{powersum} The Factorial Conjecture holds for $f\in\C[U_1,\ldots,U_n]$ of the form $c_1U_1^d+\cdots+c_nU_n^d$ where $d\ge1$.
\end{propo}

\begin{proof} The case $d=1$ is covered in Proposition \ref{formpower}, so we assume $d\ge2$ and each of $c_1,\ldots,c_n$ is non-zero.  Here we only need to assume that $\cL(f^m)=0$ for $m\gg0$.  We consider the powers $f^{nm}$ of $f$:
$$f^{nm}=\sum_{k_1+\cdots+k_n=nm}\frac{(nm)!}{k_1!\cdots k_n!}c_1^{k_1}\cdots c_n^{k_n}U_1^{k_1d}\cdots U_n^{k_nd}\,,$$
which yields
\begin{align}\cL(f^{nm})&=\sum_{k_1+\cdots+k_n=nm}\frac{(nm)!}{k_1!\cdots k_n!}c_1^{k_1}\cdots c_n^{k_n}(k_1d)!\cdots(k_nd)! \notag\\
&=(nm)! \sum_{k_1+\cdots+k_n=nm}\frac{(k_1d)!}{k_1!}\cdots\frac{(k_nd)!}{k_n!}c_1^{k_1}\cdots c_n^{k_n}\label{factsum}
\end{align}
One term of (\ref{factsum}), we'll call it the special term, occurs when $k_1=\cdots=k_n=m$.  For all other summands we have $k_i>m$ for some $i$ (since $\sum k_i=nm$), and we now examine one of these other summands.  Without loss of generality, suppose $k_1>m$ and write
\begin{align}
\frac{(k_1d)!}{k_1!}&=\frac{k_1d}{k_1}(k_1d-1)\cdots(k_1d-d+1)\frac{(k_1-1)d}{k_1-1}(k_1d-d-1)\cdots\notag\\
&\qquad\qquad\quad\cdots(2d+1)\frac{2d}{2}(2d-1)\cdots(d+1)\frac{1d}{d}(d-1)\cdots1\,.\notag
\end{align}
From this one easily sees that $\frac{(k_1d)!}{k_1!}$ is an integer divisible by $p=(m+1)d-1$, which, by Dirichlet's prime number theorem, is prime for infinitely many values of $m$.  As in previous arguments, we apply the algebraic reduction (Remark \ref{algred}) and let $\mathcal O$ be the Dedekind ring chosen as in Remark \ref{valuation}. For all but finitely many such $p$, $\mathcal O$ has a valuation $v_p$ lying over $p$.  The above observation shows then shows that $v_p$ is positive at all the terms of (\ref{factsum}) except the special term, and since $\cL(f^{nm})=0$ it must be positive at the special term as well.  Since $p=(m+1)d-1$ does not divide $\frac{(k_1d)!}{k_1!}\cdots\frac{(k_nd)!}{k_n!}$ when $k_1=\cdots=k_n=m$, we must have $v_p(c_1^m\cdots c_n^m)=0$, and since this holds for infinitely many valuations of $\mathcal O$, we conclude $c_1^m\cdots c_n^m=0$.  Therefore $c_i=0$ for some $i$, contradicting our assumption.
\end{proof}

\bibliographystyle{amsplain}
\bibliography{Refs}

\vskip 15pt
\noindent{\small \sc Department of Mathematics, Radboud University,
Nijmegen, The Netherlands } {\em E-mail}: essen@math.ru.nl

\noindent{\small \sc Department of Mathematics, Washington University in St.
Louis, St. Louis, MO 63130 } {\em E-mail}: wright@math.wustl.edu

\noindent{\small \sc Department of Mathematics, Illinois State University, Normal, IL 61790 } {\em E-mail}: wzhao@ilstu.edu

\end{document}